\pdfoutput=1
\RequirePackage{ifpdf}
\ifpdf 
\documentclass[pdftex]{sigma}
\else
\documentclass{sigma}
\fi

\usepackage{mathrsfs}
\usepackage{tikz-cd}

\newcommand{\gl}{{\mathfrak g \mathfrak l}}

\renewcommand{\u}{{\mathfrak u}}

\renewcommand{\tt}{{\mathfrak t}} 

\newcommand{\cx}{{\mathbb C}}
\newcommand{\diag}{\operatorname{diag}}

\newcommand{\re}{\operatorname{Re}}

\newcommand{\Ker}{\operatorname{Ker}}

\newcommand{\Jac}{\operatorname{Jac}}

\newcommand{\Mat}{\operatorname{Mat}}

\newcommand{\sgn}{\operatorname{sgn}}

\newcommand{\ol}{\overline}

\numberwithin{equation}{section}

\newtheorem{Theorem}{Theorem}[section]
\newtheorem*{Theorem*}{Theorem}

\newtheorem{Proposition}[Theorem]{Proposition}
 { \theoremstyle{definition}

\newtheorem{Remark}[Theorem]{Remark} }

\newcommand{\oC}{{\mathbb{C}}}

\newcommand{\oF}{{\mathbb{F}}}

\newcommand{\oH}{{\mathbb{H}}}

\newcommand{\oP}{{\mathbb{P}}}

\newcommand{\oR}{{\mathbb{R}}}

\newcommand{\oT}{{\mathbb{T}}}

\newcommand{\sF}{{\mathcal{F}}}

\newcommand{\sM}{{\mathcal{M}}} 
\newcommand{\sN}{{\mathcal{N}}}
\newcommand{\sO}{{\mathcal{O}}}

\begin{document}
\allowdisplaybreaks

\renewcommand{\thefootnote}{}

\newcommand{\arXivNumber}{2208.14936}

\renewcommand{\PaperNumber}{041}

\FirstPageHeading

\ShortArticleName{Deformations of Instanton Metrics}

\ArticleName{Deformations of Instanton Metrics\footnote{This paper is a~contribution to the Special Issue on Topological Solitons as Particles. The~full collection is available at \href{https://www.emis.de/journals/SIGMA/topological-solitons.html}{https://www.emis.de/journals/SIGMA/topological-solitons.html}}}

\Author{Roger BIELAWSKI~$^{\rm a}$, Yannic BORCHARD~$^{\rm a}$ and Sergey A.~CHERKIS~$^{\rm b}$}
\AuthorNameForHeading{R.~Bielawski, Y.~Borchard and S.~Cherkis}

\Address{$^{\rm a)}$~Institut f\"ur Differentialgeometrie, Leibniz Universit\"at Hannover,\\
\hphantom{$^{\rm a)}$}~Welfengarten 1, 30167 Hannover, Germany}
\EmailD{\href{mailto:bielawski@math.uni-hannover.de}{bielawski@math.uni-hannover.de}, \href{mailto:borchard@math.uni-hannover.de}{borchard@math.uni-hannover.de}}

\Address{$^{\rm b)}$~Department of Mathematics, The University of Arizona,\\
\hphantom{$^{\rm b)}$}~617 N. Santa Rita Ave., Tucson, AZ 85721-0089, USA}
\EmailD{\href{mailto:cherkis@math.arizona.edu}{cherkis@math.arizona.edu}}

\ArticleDates{Received January 25, 2023, in final form June 05, 2023; Published online June 13, 2023}

\Abstract{We discuss a class of bow varieties which can be viewed as Taub-NUT deformations of moduli spaces of instantons on {noncommutative} $\oR^4$. Via the generalized Legendre transform, we find the K\"ahler potential on each of these spaces.}


\Classification{53C26; 53C28; 81T13}

\begin{flushright}
\begin{minipage}{50mm}
\it To Nicholas Stephen Manton\\ on his 70th birthday
\end{minipage}
\end{flushright}

\renewcommand{\thefootnote}{\arabic{footnote}}
\setcounter{footnote}{0}

Bow varieties, introduced by the third author \cite{Che1,Che2}, are a common generalisation of quiver varieties and of moduli spaces of solutions to Nahm's equations.
A class of bow varieties describes, via an analog of the ADHM construction, moduli spaces of (framed) instantons on ALF-spaces.
In the present paper, we are interested in a very particular type of bow varieties, which can be viewed as a moduli space of $U(r)$ instantons on the {noncommutative} Taub-NUT space (cf.~Section~\ref{bow}). The case $r=1$ of these has been studied by Takayama
 \cite{Tak}. Our approach is via spectral curves and line bundles.
 This allows us to give a formula for the K\"ahler potential of the hyperk\"ahler metric via the generalised Legendre transform of Lindstr\"om and Ro\v{c}ek~\cite{LR}.
 We also derive the asymptotic metric in the region where the $U(r)$-instantons of charge $k$ can be approximated by~$kr$ well-separated constituents (cf.\ \cite[Section~9]{Che2}), which we expect to be Euclidean $U(2)$-monopoles (cf. \cite{FR}).

\section{Spectral curves, line bundles, and matrix polynomials}

The complex manifold $\oT={\rm T}\oP^1$ is equipped with the standard atlas $(\zeta,\eta)$, $(\tilde \zeta,\tilde\eta)$, where $\tilde\zeta=\zeta^{-1}$, $\tilde\eta=\eta\,\zeta^{-2}$. We recall \cite[Proposition~2.2]{AHH} that $H^1(\oT,\sO_\oT)$ is generated by monomials of the form $\eta^i\zeta^{-j}$, $i>0$, $j<2i$. Of particular interest is the line bundle $\mathscr{L}^z$, $z\in\cx$, with transition function $\exp(z\eta/\zeta)$.

A {\em spectral curve} (of degree $k$) is a compact $1$-dimensional subscheme of ${\rm T}\oP^1$ defined by the equation $P(\zeta,\eta)=0$, where $P(\zeta,\eta)=\eta^k+\sum_{i=1}^kp_i(\zeta)\eta^{k-i}$, $\deg p_i=2i$. It can be reducible or nonreduced, and its arithmetic genus $g$ is equal to $(k-1)^2$.

On a spectral curve $S$, we consider the Jacobian $\Jac^{g-1}(S)$ of line bundles $L$ (i.e.,  invertible sheaves) of degree $g-1=k^2-2k$, i.e., satisfying $\chi(L)=0$. The line bundle $\sO_S(k-2)$ has degree $g-1$, and therefore we have an isomorphism ${\rm Pic}^0(S)\to \Jac^{g-1}(S)$, $L\mapsto L(k-2)$.
As shown in \cite[Proposition~2.1]{AHH}, any line bundle on $S$ of degree zero is a restriction of a line bundle on $\oT$, and hence, the same holds for line bundles of degree $g-1$. The theta divisor $\Theta_S\subset \Jac^{g-1}(S)$ consists of line bundles with nontrivial cohomology. Beauville \cite{Beau} has shown that any $L\in \Jac^{g-1}(S)\backslash \Theta_S$, viewed as a sheaf on $\oT$, has a free resolution of the form
\begin{equation}
\begin{tikzcd}
& 0 \ar[r] &\sO_\oT(-3)^{\oplus k}\ar[r,"\eta-A(\zeta)"] &[6mm] \sO_\oT(-1)^{\oplus k}\ar[r]& L\ar[r] & 0,
\end{tikzcd}\label{res}
\end{equation}
where $A(\zeta)=A_0+A_1\zeta+A_2\zeta^2$, $A_i\in \Mat_{k,k}(\cx)$, is a quadratic matrix polynomial, the characteristic polynomial of which is $P(\zeta,\eta)$. The essential idea is that, since $\pi\colon S\to\oP^1$, $(\zeta,\eta)\mapsto\zeta$, is a finite flat morphism, and $L$ is locally free, the direct image $\pi_\ast L$ is also locally free. Since $h^0(L)=h^1(L)=0$, the same holds for $\pi_\ast L$, and so $\pi_\ast L\simeq \sO(-1)^{\oplus k}$. Moreover, $\pi_\ast L$ is a module over $\pi_\ast S$, i.e., it corresponds to a homomorphism $A\colon \pi_\ast L\to \pi_\ast L(2)$ satisfying $P(\zeta,A(\zeta))=0$. Since $L$ is a line bundle, the matrix $A(\zeta)$ is regular for every $\zeta$, and hence $P(\zeta,\eta)$ is the characteristic polynomial of $A(\zeta)$.

\begin{Remark}
The Beauville correspondence described above can be also rephrased as follows. Consider the set $Q$ of quadratic matrix polynomials $A(\zeta)$ such that $A(\zeta_0)$ is a regular matrix for every $\zeta_0\in\oP^1$. This is an open subset of $\cx^{3k^2}$ and since ${\rm GL}_n(\cx)$ is reductive, there exists a good quotient $\mathscr{J}_k=Q/{\rm GL}_k(\cx)$. This quotient, with its scheme structure, can be viewed as the universal Jacobian of spectral curves, parametrising pairs $(S,L)$, where $S$ is a spectral curve and $L\in \Jac^{g-1}(S)\backslash \Theta_S$. It can also be viewed as an open subset of Simpson's moduli space of semistable $1$-dimensional sheaves on the Hirzebruch surface $\oF_2$ with Hilbert polynomial $h(m)=km$ \cite{Simp}. \label{moduli}
\end{Remark}

\subsection{Real structures} The manifold $\oT$ is equipped with a real structure (i.e., an antiholomorphic involution) $\sigma$ defined by
\[
\sigma(\zeta,\eta)=\bigl(-1/\bar\zeta,-\bar\eta/\bar\zeta^2\bigr).
\]
If a spectral curve $S$ is real (i.e., $\sigma$-invariant), then we obtain an induced antiholomorphic involution $\sigma$ on ${\rm Pic}(S)$. This involution corresponds to complex conjugation of the matrix polynomial in \eqref{res} \cite[Section~1.2]{theta}. Since we are interested in Hermitian conjugation, we need to replace $\sigma$ by the following antiholomorphic conjugation on $\Jac^{g-1}(S)$:
\[
L\mapsto\sigma(L)^\ast\otimes \sO_S(2k-4).
\]
We denote the invariant subset of $\Jac^{g-1}(S)$ by $\Jac^{g-1}_\oR(S)$ and the corresponding subset of $\mathscr{J}_k$ (cf.\ Remark \ref{moduli})
by $\mathscr{J}_k^\oR$. A line bundle $L$ belongs to $\Jac^{g-1}_\oR(S)$ if and only if it is of the form $L_0(k-2)$, where $L_0$ is a degree $0$ line bundle with transition function $\exp q(\zeta,\eta)$ satisfying $\ol{q(\zeta,\eta)}=q(\sigma(\zeta,\eta))$.

It has been shown in \cite[Proposition~1.7]{theta} that $\mathscr{J}_k^\oR$ decomposes into disjoint subsets $\mathscr{J}_k^p$, $p=0,\dots,[k/2]$,
corresponding to standard Hermitian forms $q=-\sum_{i=1}^p |z_i|^2+\sum_{i=p+1}^k|z_i|^2$ of signature $(p,k-p)$ on $\cx^k$. Denoting by $q$ also the diagonal matrix defining the quadratic form, $\mathscr{J}_k^p$ consists of ${\rm SU}(p,k-p)$-conjugacy classes of quadratic matrix polynomials $A(\zeta)$ which satisfy
\[
qA_0q^{-1}=-A_2^\ast,\qquad qA_1q^{-1}=A_1^\ast,\qquad qA_2q^{-1}=-A_0^\ast.
\]

\begin{Remark} Equivalently, the component $\mathscr{J}_k^p$ to which a real $(S,L)$ belongs is determined by the signature of Hitchin's metric on $H^0(S,L(1))$ \cite[Section~6]{Hit}.
\end{Remark}
\begin{Remark} It is perhaps worth pointing out that, for any real spectral curve $S$ and any~$p$, $\Jac^{g-1}_\oR(S)\backslash \Theta_S$ has a nonempty intersection with $\mathscr{J}_k^p$. Indeed, for small $s\in \oR$, the line bundle $\mathscr{L}^s(k-2)|_S$ belongs to $\mathscr{J}_k^0$ (cf.~\cite[paragraph after formula~(6.11)]{Hit}). Thus the map associating to $L\in \mathscr{J}_k^0$ its support $S$ is surjective. Each $\mathscr{J}_k^p$ is, however isomorphic to $\mathscr{J}_k^0$, e.g., via $A(\zeta)\mapsto D_1A(\zeta)D_2$ for an appropriately chosen pair of diagonal matrices.
\end{Remark}

We shall be interested only in the component $\mathscr{J}_k^0$. The sheaves in this component are represented by matrix polynomials of the form
\begin{equation}
T(\zeta)=(T_2+{\rm i} T_3)+2{\rm i} T_1\zeta+(T_2-{\rm i} T_3)\zeta^2,\qquad T_i\in\u(k),\label{T(zeta)}
\end{equation}
modulo conjugation by $U(k)$.
As in \cite{theta}, we shall call sheaves belonging to $\mathscr{J}_k^0$ {\em definite}.

\subsection{Nahm's equations}
$\Jac^{g-1}(S)$ is a torsor for ${\rm Pic}^0(S)$. Therefore the tangent bundle of $\Jac^{g-1}(S)$ is parallelisable and canonically isomorphic to $ \Jac^{g-1}(S)\times H^1(S,\sO_S)$. If we choose an element of $ H^1(S,\sO_S)$, we obtain a linear flow on $\Jac^{g-1}(S)$. Restricting this flow to the complement of the theta divisor, and choosing an appropriate connection (cf.\ \cite{Hit} and \cite{AHH}) yields a flow of quadratic matrix polynomials corresponding to elements of $\Jac^{g-1}(S)\backslash \Theta_S$. In particular, for the flow given by $[\eta/\zeta]\in H^1(S,\sO_S)$, i.e., $L\mapsto L\otimes \mathscr{L}^z$, there is a connection such that the restriction of the flow to $z\in \oR$ and to the definite line bundles
(i.e., to matrix polynomials of form \eqref{T(zeta)}) is given by
\begin{equation*}
\frac{\partial T(\zeta)}{\partial z}=\frac{1}{2}\bigg[T(\zeta),\frac{\partial T(\zeta)}{\partial \zeta}\bigg],
\end{equation*}
which is equivalent to Nahm's equations
\begin{equation}\label{eq:Nahm}
\dot{T_i}+\frac{1}{2}\sum_{j,k}\epsilon_{ijk}[T_j.T_k]=0,\qquad
i=1,2,3.
\end{equation}

\section{Factorisation of matrix polynomials}
We consider the flat hyperk\"ahler manifold $T^\ast \Mat_{k,k}(\cx)$, which we identify with $\Mat_{k,k}(\cx)\oplus \Mat_{k,k}(\cx)$.
It has a natural tri-Hamiltonian $U(k)\times U(k)$-action given by
\[
(g,h).(A,B)=\big(gAh^{-1},hBg^{-1}\big),
\]
and the corresponding hyperk\"ahler moment maps are:
\begin{alignat*}{3}
& (\mu_2+{\rm i}\mu_3)(A,B)=AB,\qquad&& 2{\rm i}\mu_1(A,B)=AA^\ast-B^\ast B,&
\\
&(\nu_2+{\rm i}\nu_3)(A,B)=-BA,\qquad && 2{\rm i}\nu_1(A,B)=BB^\ast-A^\ast A.&
\end{alignat*}
We can view these moment maps as sections of $\sO(2)\otimes \gl_k(\cx)$ over the $\oP^1$ parametrising complex structure, and write them as quadratic matrix polynomials:
\begin{gather}\mu(\zeta)=(A-B^\ast\zeta)(B+A^\ast\zeta),\label{mu}
\\
 \nu(\zeta)=-(B+A^\ast\zeta)(A-B^\ast\zeta).\label{nu}
 \end{gather}
As explained in the previous section $\mu(\zeta)$ and $-\nu(\zeta)$ define $1$-dimensional sheaves $\sF$, $\sF^\prime$ in~$\mathscr{J}_k^0$
(i.e., real, acyclic, and definite). Moreover, $\sF$ and $\sF^\prime$ are supported on the same spectral curve~$S$. Our first goal is to relate $\sF^\prime$ to $\sF$. Since we do not need the reality conditions for this, let us consider arbitrary linear matrix polynomials $A(\zeta)$, $B(\zeta)$, such that the roots of $\det A(\zeta)$ are disjoint from the roots of $\det B(\zeta)$. Let $\sF\in \mathscr{J}_k$ (resp.~$\sF^\prime\in \mathscr{J}_k$) be the sheaf determined by $A(\zeta)B(\zeta)$ (resp.~by $B(\zeta)A(\zeta)$). Let $S$ be the common support of $\sF$ and $\sF^\prime$, and let $\Delta_A$ (resp.~$\Delta_B$) be the Cartier divisor on $S$ given by $\eta=0$ on the open subset $\det B(\zeta)\neq 0$ (resp.~on the open subset $\det A(\zeta)\neq 0$).
\begin{Proposition} $\sF^\prime\simeq \sF(1)[-\Delta_A]$.
\label{sF'}\end{Proposition}
\begin{proof} We have a commutative diagram
\[
\begin{tikzcd} & 0 \ar[d] &[30pt] 0 \ar[d]& 0\ar[d] &\\[-6pt] 0 \ar[r] & \sO_{\oT}(-3)^{\oplus k}\ar[r, "\eta-A(\zeta)B(\zeta)"] \ar[d, "B(\zeta)"] &\sO_\oT(-1)^{\oplus k} \ar[d, "B(\zeta)"] \ar[r]& \sF\ar[r]\ar[d] & 0\\[3pt]
0 \ar[r] &\sO_\oT(-2)^{\oplus k}\ar[r, "\eta-B(\zeta)A(\zeta)"]\ar[d]& \sO_\oT^{\oplus k}\ar[r] \ar[d] &\sF^\prime(1)\ar[r]\ar[d] & 0\\ 0\ar[r] &C \ar[d]\ar[r, "\eta"]& C(2) \ar[d]\ar[r] & \sO_{\Delta_B}\ar[d]\ar[r] & 0,\\[-6pt] & 0 & 0 & 0 &
\end{tikzcd}
\]
where $C$ is the cokernel of $B(\zeta)$.
Therefore, $\sF^\prime(1)\simeq \sF[\Delta_B]$. Since $[\Delta_A+\Delta_B]\simeq \sO_S(2)$, the claim follows.
\end{proof}

We now ask whether a given quadratic polynomial $T(\zeta)$, corresponding to a sheaf in $\mathscr{J}_k^0$,
can be factorised as in formula~\eqref{mu}. Generically, the answer is yes.

\begin{Proposition} Let $T(\zeta)$ be of form \eqref{T(zeta)} and suppose that
\begin{enumerate}\itemsep=0pt
\item[$(i)$] the polynomial $\det T(\zeta)$ has $2n$ distinct zeros $\zeta_1,\dots,\zeta_{2n}$,
\item[$(ii)$] the corresponding eigenvectors $v_i\in \Ker T(\zeta_{i})$, $i=1,\dots,2n$, are in general position, i.e., for any choice $i_1<\dots<i_n\in \{1,\dots,2n\}$, $v_{i_1},\dots,v_{i_n}$ are linearly independent.
\end{enumerate}
Then $T(\zeta)$ can be factorised as $(A-B^\ast\zeta)(B+A^\ast\zeta)$.\label{distinct}
\end{Proposition}

\begin{proof}
After rotating $\oP^1$, we can assume that $\zeta=\infty$ is not a root of $\det T(\zeta)$. Let $\Delta\cup \sigma(\Delta)$ be a decomposition of the set of zeros of $\det T(\zeta)$. Theorem 1 in \cite{Mal} implies that there is a~decomposition $T(\zeta)=(C_1+D_1\zeta)(C_2+D_2\zeta)$ such that $\Delta$ is the set of roots of $\det(C_2+D_2\zeta)$. Applying the real structure shows that $(D_2^\ast-C_2^\ast\zeta)(-D_1^\ast+C_1^\ast\zeta)$ is also a factorisation of $T(\zeta)$.
We can rewrite these factorisations as
\[
T(\zeta)=\big(C_1D_1^{-1}+\zeta\big)(D_1C_2+D_1D_2\zeta)=\bigl(-D_2^\ast(C_2^\ast)^{-1}+\zeta\big)(C_2^\ast D_1^\ast-C_2^\ast C_1^\ast\zeta).
\]
Theorem 2 in \cite{Mal} implies now that $C_1D_1^{-1}=-D_2^\ast(C_2^\ast)^{-1}$, i.e., $D_1^{-1}C_2^\ast=-C_1^{-1}D_2^\ast$. In addition, comparing the constant coefficients of the two factorisations, we have $C_1C_2=-D_2^\ast D_1^\ast$. Hence
\[
\big(D_1^{-1}C_2^\ast\big)^\ast=C_2(D_1^\ast)^{-1}=-C_1^{-1}D_2^\ast D_1^\ast(D_1^\ast)^{-1}= -C_1^{-1}D_2^\ast=D_1^{-1}C_2^\ast.
\]
Therefore, $D_1^{-1}C_2^\ast$ is hermitian (and invertible). We can write it as $-gdg^\ast$, where $g$ is invertible and $d$ is diagonal with diagonal entries equal $\pm 1$.
Then
\begin{equation} T(\zeta)=(C_1+D_1\zeta)gg^{-1}(C_2+D_2\zeta)=(C_1g+D_1g\zeta)d(-g^\ast D_1^\ast+g^\ast C_1^\ast\zeta).
\label{CDf}\end{equation}
The uniqueness of monic factors of $T(\zeta)$ implies that the map $\Delta\mapsto d$ is injective. Since both sets have the same cardinality (equal to $2^k$), this map is surjective, i.e., there is a choice of $\Delta$ such that the corresponding $d$ is the identity matrix, and formula~\eqref{CDf} becomes the desired factorisation.
\end{proof}

\section{Deformed instanton moduli spaces\label{bow}}

We consider a bow variety $\sM$ corresponding to the bow {representation} diagram {in Figure~\ref{fig:Bowdiagram}}:
\begin{figure}[h]\centering
\includegraphics[width=0.45\textwidth]{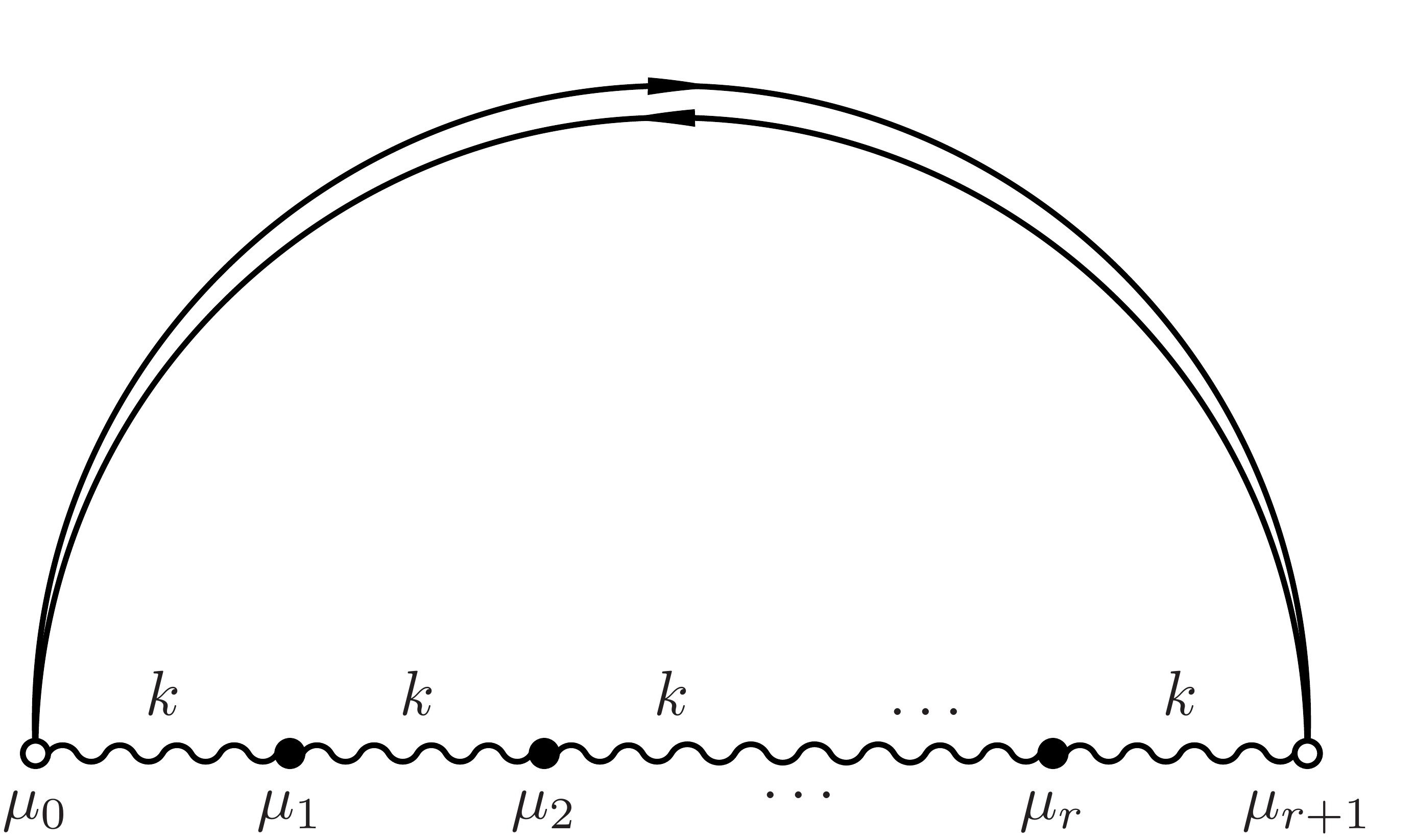}
\put(-111,114){\makebox(0,0)[lb]{\small$A$}}
\put(-112,94){\makebox(0,0)[lb]{\small$B$}}
	\caption{Bow representation diagram with $r$ $\lambda$-points $\mu_1,\ldots,\mu_r$ and constant rank $k$.}	\label{fig:Bowdiagram}
\end{figure}	
\noindent
with $r$ $\lambda$-points and the rank of {all bundles} equal to $k$.
In other words, $\sM$ is the moduli space of $\u(k)$-valued solutions to Nahm's equations on $[\mu_0,\mu_{r+1}]$ which have rank $1$ {discontinuity in $(T_2+{\rm i} T_3)+2{\rm i} T_1\zeta+(T_2-{\rm i} T_3)\zeta^2$ at each $\mu_i$, $i=1,\dots,r$, and
$(T_2+{\rm i} T_3)+2{\rm i} T_1\zeta+(T_2-{\rm i} T_3)\zeta^2$ is equal to $(B+A^\ast\zeta)(A-B^\ast\zeta)+c_L(\zeta){\rm Id}$ at $\mu_0$ and to $(A-B^\ast\zeta)(B+A^\ast\zeta)+c_R(\zeta){\rm Id}$ at $\mu_{r+1}$, where $A,B\in \Mat_{k,k}(\cx)$ and $c_L$, $c_R$ are quadratic polynomials satisfying the reality condition.

Let us consider two limiting cases.

First, is the case when we let the lengths of all intervals go to zero, then $\sM$ is the quotient by $U(k)$ of the set of solutions to the following matrix equations:
\[
[A-B^\ast\zeta,B+A^\ast\zeta]=\sum_{i=1}^r (v_i-\bar{w}_i\zeta)(w_i+\bar{v}_i\zeta)^{\rm T}+(c_L(\zeta)-c_R(\zeta)),
\]
where $v_i,w_i\in \cx^k$.
In particular, if $c_L(\zeta)-c_R(\zeta)=a\zeta$, then $\sM$ with the complex structure corresponding to $\zeta=0$ is biholomorphic to the moduli space of framed torsion-free sheaves on~$\oP^2$ with rank $r$ and $c_2=k$ \cite[Theorem~2.1]{Nak}. For an arbitrary nonzero $(c_L(\zeta)-c_R(\zeta))$, $\sM$ (with $\mu_0=\dots=\mu_{r+1}$) has been interpreted by Nekrasov and Schwarz as a moduli space of instantons on a noncommutative $\oR^4$ \cite{Nekr}.
We can, therefore, view $\sM$ with arbitrary $\mu_i$ as a {\em deformation} of the moduli space of instantons on noncommutative $\mathbb{R}^4$ with the noncommutativity parameter $c_L(\zeta)-c_R(\zeta)$.
It changes the space geometry from a higher-dimensional ALE to ALF kind, as we explain in the beginning of Section~\ref{Sec:Asymp}. For $r=1$, these moduli spaces have been investigated in detail by Takayama~\cite{Tak}.

We remark that the hyperk\"ahler metric on our $\sM$ has a $T^r$-symmetry, compared to a $U(r)$-symmetry of the moduli space of instantons on the noncommutative $\mathbb{R}^4.$

Second, in the case with $c_L(\zeta)=c_R(\zeta)$, $\sM$ is isometric to the moduli space of instantons on the Taub-NUT space \cite{CLHS}. Notably, while the deformation to nonzero $c_L(\zeta)-c_R(\zeta)$ appears rather benign from the moduli space point of view, it is nearly fatal to the ADHM-type transform from the bow to the instanton, since the corresponding {\em small bow representation} moduli space becomes empty, instead of being the Taub-NUT space. This is completely analogous to the situation with the original ADHM construction and its noncommutative deformation of Nekrasov and Schwarz.

\subsection{Complex structures\label{complex}}
We shall now show that the complex-symplectic structures of $\sM$ do not depend on the $\mu_i$ (this has been shown by Takayama for $r=1$). First of all, $\sM$ is isomorphic to a hyperk\"ahler quotient of $\tilde \sM\times T^\ast \Mat_{k,k}(\cx)$ by $U(k)\times U(k)$, where $\tilde\sM$ is the moduli space of solutions to Nahm's equations on $r+1$ intervals as above, without the bifundamental representation, i.e., without the half-circles labelled by $A$ and $B$. We discuss first the complex-symplectic structures of~$\tilde\sM$. Let us consider the complex structure $I$ corresponding to $0\in \oP^1$ (all complex structures of $\tilde\sM$ are isomorphic). We can, following Donaldson \cite{Don}, separate the data given by Nahm's equations and boundary conditions, into a complex and a real part.
The complex part is given by solutions of the Lax equation $\dot\beta=[\beta,\alpha]$ on each interval {$[\mu_i,\mu_{i+1}]$}, where $\beta(t)=T_2(t)+{\rm i} T_3(t)$, $\alpha(t)={\rm i} T_1(t)$ with rank $1$ discontinuity at $\mu_1,\dots,\mu_r$.
It follows from results of Donaldson~\cite{Don} and Hurtubise~\cite{Hurt} that $\tilde\sM$ is biholomorphic to the quotient of this space by ${\rm GL}(k,\cx)$-valued gauge transformations which are identity at $\mu_0$ and $\mu_{r+1}$ and match at the remaining~$\mu_i$.
This biholomorphism preserves also the complex-symplectic form.
On each interval one can apply a complex gauge transformation to make $\alpha$ identically zero and $\beta$ constant. If we do this
beginning with the left-most interval and such a gauge transformation with $g(\mu_0)=1$, we can make $\beta(t)$ equal to a constant $\beta_i$ on each $[\mu_{i-1},\mu_i]$, $i=1,\dots, r+1$, with $\beta_{i+1}-\beta_i=I_iJ_i$ for a~vector $I_i$ and a covector $J_i$. The map associating to $(\beta(\mu_0),g(\mu_{r+1}), I,J)$, where $I=[I_1,\dots,I_r]$ and $J=[J_1,\dots, J_r]^{\rm T}$ to a point of $\tilde\sM$ is a complex-symplectic isomorphism between $\tilde\sM$ and $T^\ast {\rm GL}(k,\cx)\times T^\ast\Mat_{k,r}$.

The complex-symplectic quotient of the product of $T^\ast {\rm GL}(k,\cx)\times T^\ast\Mat_{k,r}$ and $T^\ast \Mat_{k,k}(\cx)$ by ${\rm GL}(k,\cx)\times {\rm GL}(k,\cx)$ (which is the remaining gauge freedom at $\mu_0$ and $\mu_{r+1}$)
 can be performed in two stages: the quotient by the left copy of ${\rm GL}(k,\cx)$ (the one which acts trivially on $I$ and~$J$) is $T^\ast\Mat_{k,r}\times T^\ast \Mat_{k,k}(\cx)$.
The remaining symplectic quotient is the
same one as in the case with
$\mu_0=\dots=\mu_{r+1}$. This shows that, as long as $c_L(\zeta)-c_R(\zeta)\neq 0$, $\sM$ is isomorphic, as~a~complex-symplectic manifold,
to the corresponding space of noncommutative instantons.

\subsection{Spectral curves}
We shall now describe the moduli space $\sM$ using the language of spectral curves and line bundles.
 We denote by $S_i$ the spectral curve on the interval $[\mu_{i},\mu_{i+1}]$. Due to the matching conditions, $S_{r}$ is equal to $S_0$ shifted by $\eta\mapsto \eta+ c(\zeta)$, where $c(\zeta)=c_L(\zeta)-c_R(\zeta)$.
\par
Hurtubise and Murray \cite{HM} analysed what happens to spectral curves and line bundles at rank $1$ {discontinuity} of solutions to Nahm's equations. Namely, for $i=0,\dots,r-1$, we have $S_i\cap S_{i+1}=D_{i,i+1}\cup D_{i+1,i}$ with $\sigma(D_{i,i+1})=D_{i+1,i}$ and the line bundles at $\mu_{i+1}$ equal to $\sO_{S_i}(2k)[-D_{i,i+1}]\in \Jac^{g-1}(S_i)$, $\sO_{S_{i+1}}(2k)[-D_{i,i+1}]\in \Jac^{g-1}(S_{i+1})$. It follows that $S_1,\dots,S_{r-1}$ satisfy the following condition
\begin{equation} \mathscr{L}_{S_i}^{\mu_{i+1}-\mu_{i}}[D_{i,i+1}-D_{i-1,i}]\simeq \sO_{S_i}.\label{triviality}
\end{equation}
It remains to identify the condition satisfied by $S_0$ and $S_{r}$. The line bundles at $\mu_0$ and
at $\mu_{r+1}$ are
 $\mathscr{L}_{S_0}^{\mu_0-\mu_1}(2k)[-D_{0,1}]$  and $\mu_{r+1}$ is $\mathscr{L}_{S_{r}}^{\mu_{r+1}-\mu_r}(2k)[-D_{r-1,r}]$, respectively.
 For any quadratic polynomial $c=c(\zeta)$ denote by $\phi_c$ the automorphism of $\oT={\rm T}\oP^1$ given by $\eta\mapsto \eta+ c(\zeta)$. The induced map on $H^1(\oT,\sO_{\oT})$ is trivial. Let us denote by $S_c$ the image of $S_0$ under $\phi_{c_L}$ (equivalently, the image of $S_{r}$ under $\phi_{c_R}$). It follows that $B(\zeta)A(\zeta)$ represents the line bundle $\mathscr{L}_{S_c}^{\mu_0-\mu_1}(2k)[-\phi_{c_L}(D_{0,1})]$ and $A(\zeta)B(\zeta)$ represents the line bundle $\mathscr{L}_{S_c}^{\mu_{r+1}-\mu_r}(2k)[-\phi_{c_R}(D_{r-1,r})$.
 Proposition \ref{sF'} implies that
\[
\mathscr{L}_{S_c}^{\mu_0-\mu_1}(2k)\bigl[-\phi_{c_L} (D_{0,1} )\bigr]\simeq \mathscr{L}_{S_c}^{\mu_{r+1}-\mu_r}(2k)\bigl[-\phi_{c_R} (D_{r-1,r} )\bigr]\otimes \sO_{S_c}(1)[-\Delta_A],
\]
that is,
 \begin{equation} \mathscr{L}_{S_c}^{\mu_{r+1}-\mu_r+\mu_1-\mu_0}(1)\bigl[\phi_{c_L} (D_{0,1} )-\phi_{c_R} (D_{r-1,r} )-\Delta_A\bigr]\simeq \sO_{S_c},\label{triviality2}
 \end{equation}
 where $\Delta_A$ is be the divisor on $S_c$ cut out by $\eta=0$ on the open subset $\det B(\zeta)\neq 0$ (thus $\det A(\zeta)=0$ on $\Delta_A$).
In addition, the spectral curves $S_c,S_1,\dots, S_{r-1}$ satisfy appropriate nondegeneracy conditions, which simply mean that the flow of line bundles on each $S_i$ does not meet the theta divisor. Conversely, given generic curves $S_c,S_1,\dots, S_{r-1}$ satisfying these conditions together with trivialisations in the formulas~\eqref{triviality} and \eqref{triviality2}, we obtain, using the results of \cite{HM} and Proposition \ref{distinct}, a unique gauge equivalence class of solutions to Nahm's equations in $\sM$. Here ``generic" means that $S_i \cap S_{i+1}$ for $i=0,\dots,r-1$ as well as $S_c\cap \{\eta=0\}$ consist of distinct points.

 \subsection{Generalised Legendre transform}
 The complex symplectic quotient described in Section~\ref{complex} can be performed for each complex structure, i.e., on the fibres of the twistor space of $\tilde \sM\times T^\ast \Mat_{k,k}(\cx)$. The spectral curves and (real) trivialisations of line bundles \eqref{triviality} and \eqref{triviality2} provide twistor lines corresponding to an open dense subset of $\sM$. In particular, for each complex structure, the roots of polynomials defining spectral curves and values of trivialising sections of line bundles \eqref{triviality}--\eqref{triviality2} define Darboux coordinates for the corresponding complex-symplectic form. This picture is a particular case of the generalised Legendre transform construction of Lindstr\"om and Ro\v{c}ek \cite{HKLR,LR}, which we now recall.

 The generalised Legendre transform describes $4n$-dimensional hyperk\"ahler metrics, the twistor space $Z^{2n+1}$ of which admits a projection to the total space of a vector bundle $E=\bigoplus_{i=1}^n\sO(2k_i)$ over $\oP^1$, $k_i\geq 1$, $i=1,\dots,n$. The projection is required to commute with real structures and its fibres for each $\zeta\in \oP^1$ are Lagrangian for the fibre-wise complex symplectic form on $Z^{2n+1}$. The hyperk\"ahler structure is then defined on a subset $M$ of real sections of $E$ consisting of those $\alpha_i(\zeta)=\sum_{a=0}^{2k_i}w_{ia}\zeta^a$, $i=1,\dots, n$, which satisfy
 \begin{align}\label{eq:constr}
 F_{w_{ia}}&:=\frac{\partial F}{\partial w_{ia}}=0\qquad \text{for}\quad a=2,\dots, 2k_i-2,
 \end{align}
 for a function $F$ defined as a contour integral
 \begin{equation}
 F(w_{ia})=\oint_c G\bigl(\zeta,\alpha_1(\zeta),\dots,\alpha_n(\zeta)\bigr) \frac{{\rm d}\zeta}{\zeta^2}.\label{FG}
 \end{equation}
 Complex coordinates on $\sM$ with respect to the complex structure corresponding to $\zeta=0$ are given by $z_i=w_{i0}$, $i=1,\dots, n$, and by $u_i$, where $ u_i=F_{w_{i1}}$ if $k_i\geq 2$ and $u_i+\ol u_i=F_{w_{i1}}$ if $k_i=1$.
The other coefficients $w_{ia}$ with $a>0$ are understood to be functions of $\{z_i, u_i\}$ determined by equations~\eqref{eq:constr}.
The K\"ahler potential is given by $K=F-2\sum_{i=1}^n\re u_iw_{i1}$.

 In the case of our bow variety $\sM$, $E=\bigoplus_{i=1}^k \sO(2i)^{\oplus r}$ with the summands corresponding to coefficients of powers of $\eta$ in the polynomials defining the spectral curves $S_c,S_1,\dots,S_{r-1}$.
 It has been shown in \cite{BielGZ} that conditions such as \eqref{triviality} and
 \eqref{triviality2} on spectral curves correspond to a particular choice of the function $G$ and the contour $c$ in formula~\eqref{FG}.
 In fact, one can replace the usually multi-valued function $G$ with a single-valued function on a branched cover of $\oP^1$. This cover is precisely the union of spectral curves $S_c\cup S_1\cup\cdots\cup S_{r-1}$. Although it is not necessary (as long as we allow integration over chains rather than contours), it is better to enlarge this cover by the fixed projective line $\eta=0$ (the integration contour will enter this line from $S_c$ at points of $\Delta_B$ and leave it at points of $\Delta_A$).

 In order to have trivialising sections satisfying {assumptions} of~\cite[Theorem~7.5]{BielGZ} (cf.\ Example~8.2 there), we need to replace a nonvanishing section $s_i$ of the left-hand side in formula~\eqref{triviality} by $s_i/\ol{\sigma^\ast s_i}$, which is a section of
\[
\mathscr{L}_{S_i}^{2(\mu_{i+1}-\mu_{i})}[D_{i,i+1}+D_{i,i-1}-D_{i+1,i}-D_{i-1,i}].
\]
 Similarly, we obtain from formula~\eqref{triviality2} a section of
\[
\mathscr{L}_{S_c}^{2(\mu_{r+1}-\mu_r+\mu_1-\mu_0)}\bigl[\phi_{c_L} (D_{0,1} ) +\phi_{c_R} (D_{r,r-1} )+\Delta_B -\phi_{c_L} (D_{1,0} ) -\phi_{c_R} (D_{r-1,r} )-\Delta_A\bigr].
\]
 The assumptions of \cite[Theorem~7.5]{BielGZ} are now satisfied, and we can conclude from it that the hyperk\"ahler metric on $\sM$ is given by the generalised Legendre transform applied to {the function~$F(w_{ia})$ given by}
\[
\oint_\gamma \frac{\eta}{2\zeta^2}\,{\rm d}\zeta
-\frac{1}{2\pi {\rm i}}\sum_{i=1}^{r-1}(\mu_{i+1}-\mu_{i})\oint_{\tilde 0_i}\frac{\eta^2}{2\zeta^3}\,{\rm d}\zeta-\frac{1}{2\pi {\rm i}}(\mu_{r+1}-\mu_r+\mu_1-\mu_0)\oint_{\tilde 0_c}\frac{\eta^2}{2\zeta^3}\,{\rm d}\zeta,
\]
 where $\tilde 0_i$ (resp.~$\tilde 0_c$) is the sum of simple contours around points in $S_i$ (resp.~in $S_c$) lying over ${0\in\oP^1}$, while $\gamma$ is a contour which enters (resp.~leaves) each $S_i$, $i=2,\dots,r$, at points of $D_{i+1,i}+D_{i-1,i}$
 (resp.~$D_{i,i+1}+D_{i,i-1}$), and it enters (resp.~leaves) $S_c$ at points of $\phi_{c_L}(D_{0,1}) +\phi_{c_R}(D_{r-1,r})+\Delta_A$ (resp.~$\phi_{c_L}(D_{1,0})+\phi_{c_R}(D_{r,r-1})+\Delta_B$).

 \section{Asymptotic metrics}\label{Sec:Asymp}
 In the case $\mu_0=\dots=\mu_{r+1}$, the hyperk\"ahler metric on $\sM$ has Euclidean volume growth (i.e., proportional to $ R^{4kr}$) and it is asymptotic to a Riemannian cone on a singular $3$-Sasakian manifold. Allowing the length of $m$ of the $r$ intervals $[\mu_i,\mu_{i+1}]$ to be positive, reduces the volume growth {power exponent by} $mk$. In particular, if $\mu_{i+1}-\mu_i>0$ for every $i=0,\dots,r$, then the volume growth is
 proportional to $R^{3kr}$. In this section, we shall show that, on an open dense subset, the metric is asymptotic to the Lee--Weinberg--Yi metric \cite{LWY}.

 The basic idea is the same as in \cite{BSU}: the functions $\hat{T}_i(t)=\epsilon T_i(\epsilon t)$ satisfy the same Nahm equations \eqref{eq:Nahm} as the original $T_i(t)$. Thus, exploring infinity of $\sM$ is equivalent to studying finite $\hat{T}_i$ on rescaled long intervals.
Under such rescaling, the lengths of the intervals go to infinity and we can consider a hyperk\"ahler manifold ``glued together" from $r$ moduli spaces of solutions to Nahm's equations on $\oR$ with a rank $1$ {discontinuity} at $t=0$, plus diagonal matrices~$A$,~$B$. The resulting hyperk\"ahler metric will be the asymptotic metric in the region of $\sM$ where spectral curves degenerate to unions of lines. Let us recall from \cite{BSU} the precise definition of these building blocks.

\subsection{Building blocks}
Let $a_-$, $a_+$ be positive real numbers. We shall denote\footnote{These were denoted by $\tilde F_{k,k}(a,a^\prime)$ in \cite{BSU}.} by $\sN_k(a_-,a_+)$ the moduli space of $\u(k)$-valued solutions {$(T_0(t),T_1(t),T_2(t),T_3(t))$} to Nahm's equations on $\oR$ satisfying the following conditions:
 \begin{itemize}\itemsep=0pt
 \item The solutions are analytic on $(-\infty, 0]$ and on $[0,\infty)$. At $t=0$, there is a rank one discontinuity, i.e., there exist vectors $I,J^\ast \in\oC^k$ such that $(T_2+i T_3)(0_+)-(T_2+i T_3)(0_-)=IJ$
and $T_1(0_+)-T_1(0_-)=\frac{1}{2}(II^\ast-J^\ast J).$
 	\item The $\hat{T}_i$ approach exponentially fast a diagonal limit as $t\rightarrow \pm\infty$ with $(T_1(-\infty),T_2(-\infty),\allowbreak T_3(-\infty))$ and $(T_1(+\infty),T_2(+\infty), T_3(+\infty))$ regular triples, i.e., the centraliser of the triple consists of diagonal matrices.
 	\item The gauge group has a Lie algebra consisting of functions $\rho\colon \oR\rightarrow\mathfrak{u}(k)$ such that:
 	\begin{enumerate}\itemsep=0pt
 		\item[(1)] $\rho(0)=0$ and $\dot{\rho}$ has a diagonal limit at $t\rightarrow \pm\infty$,
 		\item[(2)] $(\dot{\rho}-\dot{\rho}(+\infty))$ and $[\rho,\tau]$ decay exponentially fast for any regular diagonal matrix $\tau\in\mathfrak{u}(k)$, and similarly at $t=-\infty$,
 		\item[(3)] $a_+\dot{\rho}(+\infty)+\lim_{t\rightarrow +\infty}(\rho(t)-t\dot{\rho}(+\infty))=0$,
 		and similarly at $t=-\infty$.
 	\end{enumerate}
 \end{itemize}
 Let us denote by ${\bf x}^+_i$ (resp.~${\bf x}^-_i$) the $i$-th diagonal entry of the triple {$(T_1(+\infty),\allowbreak T_2(+\infty),\allowbreak T_3(+\infty))$} (resp.~$(T_1(-\infty),T_2(-\infty),T_3(-\infty))$).
The collection $\{{\bf x}_i^+\}_{i=1}^k$ of $k$ triplets (as well as $\{{\bf x}_i^-\}_{i=1}^k$)
might be viewed as $k$ points of $\oR^3$.
 As shown in \cite{BSU}, $\sN_k(a_-,a_+)$ is a hyperk\"ahler\footnote{Strictly speaking the metric is positive-definite only in an asymptotic region.} manifold, which topologically is a torus bundle over $\tilde C_k\big(\oR^3\big)\times \tilde C_k\big(\oR^3\big) $, where $\tilde C_k\big(\oR^3\big)$ denotes the configuration space of $k$ distinct and distinguishable points in~$\oR^3$. The action of the torus $T^k\times T^k$ is tri-Hamiltonian and the hyperk\"ahler moment map is given by ${\bf x}^-_i$, ${\bf x}^+_i$, $i=1,\dots,k$.
 Let us write~${\bf x}^{-i}$ for~${\bf x}^-_i$, ${\bf x}^{i}$ for ${\bf x}^+_i$, and ${\bf x}_\nu\in \oR^{2k}$, $\nu=1,2,3$, for the vector of $\nu$-coordinates of the ${\bf x}^i$, $|i|=1,\dots,k$. The metric is given by the Gibbons--Hawking ansatz, i.e., it is of the form
 \begin{equation}
 \sum_{\nu=1}^3 d{\bf x}_\nu^{\rm T}\Phi d{\bf x_\nu} +(dt+A)^{\rm T}\Phi^{-1}(dt+A),\label{metric}
 \end{equation}
 where $dt$ is the diagonal matrix of $1$-forms dual to Killing fields, $A$ is a connection $1$-form, and the matrix $\Phi$ (which determines the metric up to gauge equivalence) is given explicitly by
 \begin{equation*}
 \Phi_{ij}=\begin{cases}\displaystyle a_{\sgn(i)}+\sum_{k\neq i}\frac{s_{ik}}{\|{\bf x}^i-{\bf x}^k\|} & \text{if $i=j$},
 \\
 \displaystyle -\frac{s_{ij}}{\|{\bf x}^i-{\bf x}^j\|} & \text{if $i\neq j$},\end{cases}
 \end{equation*}
 where $s_{ij}=-\sgn(i)\sgn(j)$.

 There is one more building block, corresponding to matrices $A$, $B$. In our asymptotic region, these will become almost diagonal, so that this building block is $\oH^k$ with its standard flat metric and the diagonal torus action.

\subsection{Asymptotic coordinates and metric}
 \begin{figure}[ht]\centering
	\includegraphics[width=0.9\textwidth]{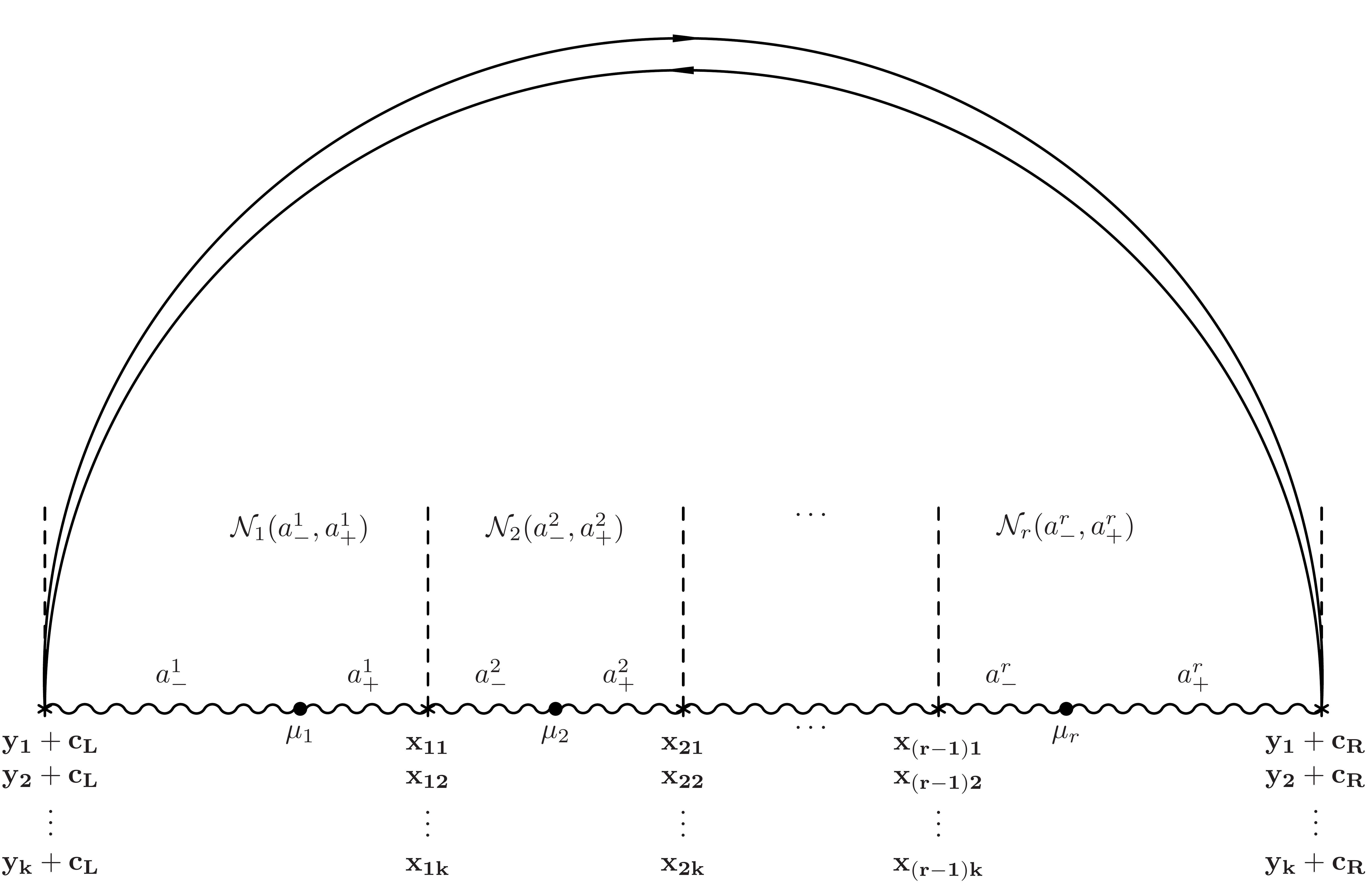}
\put(-211,260){\makebox(0,0)[lb]{\small$A$}}
\put(-212,236){\makebox(0,0)[lb]{\small$B$}}
	\caption{Bow asymptotic as hyperk\"ahler reduction of the approximation blocks. The bow interval is cut at crosses into subintervals, each containing a single $\lambda$-point $\mu_i$ with length $a_i^-$ to the left of $\mu_i$ and length $a_i^+$ to its right. The corresponding approximation space is $\sN_k(a_-^{i},a_+^i)$.}	
\end{figure}	

 We now obtain the asymptotic metric, analogously to \cite{BSU}, by gluing together these building blocks, i.e., performing the hyperk\"ahler quotient with respect to the torus.

 We start with the product $\prod_{i=1}^r\sN_k(a_-^{i},a_+^i) \times \oH^k$ with $a_+^i+a_-^{i+1}=\mu_{i+1}-\mu_i$ for $i=0,\dots,r$, where $a_+^0=a_-^{r+1}=0$. This hyperk\"ahler manifold has, as explained above, a tri-Hamiltonian action of $T^k\times T^k$ on each of the first $r$ factors and of $T^k$ on the last factor. Let us denote the torus $T^k\times T^k$ acting on $\sN_k(a_-^{i},a_+^i)$ by $T_i^-\times T_i^+$, where $T_i^-$ (resp.~$T_i^+$) is given by gauge transformations
 asymptotic to $\exp(a^i_\pm h-th)$ as $t\to -\infty$ (resp.~$t\to +\infty$), with $h\in \u(k)$. Let us also write $T^+_0$ for the standard torus action $(t,q)\mapsto\phi(t,q)$ on $\oH^k$, and $T^-_{r+1}$ for the action $(t,q)\mapsto \phi\big(t^{-1},q\big)$. We now perform the hyperk\"ahler quotient with respect to $\big(T^k\big)^{r+1}$, the $i$-th factor of which is embedded diagonally into $T_i^+\times T_{i+1}^-$,
 $i=0,\dots,r$. The level set of the hyperk\"ahler moment map is $(c_L,\dots,c_L)$ for the first copy of $T^k$, by $(c_R,\dots,c_R)$ for the last copy, and is equal to $0$ for all others (where $c_L$, $c_R$ are points in $\oR^3$ determined by the quadratic polynomials $c_L(\zeta),c_R(\zeta)$ used to define the bow variety $\sM$).

 The resulting metric is again of the form \eqref{metric}, where this time we have $kr$ points ${\bf x}^i\in \oR^3$: $k$~for each of the middle $r-1$ intervals and $k$ given by the moment map on each copy of $\oH$. Let us denote by ${\bf x}_{ij}$, $j=1,\dots,k$ the points corresponding to the interval
 $[\mu_i,\mu_{i+1}]$, $i=1,\dots,r-1$. Each ${\bf x}_{ij}$ is equal to ${\bf x}_j^+$ for $\sN_k(a_-^{i},a_+^i)$ and also to ${\bf x}_j^-$ for $\sN_k\big(a_-^{i+1},a_+^{i+1}\big)$. Let us also write ${\bf y}_1,\dots,{\bf y_k}\in\oR^3$ for the coordinates on each $\oH\backslash\{0\}$ given by the hyperk\"ahler moment map. The metric on $\oH$ can be also written in the form \eqref{metric} with $\Phi=\|{\bf y}\|^{-1}.$ Observe that ${\bf x}_j^-$ for $\sN_k\big(a_-^{1},a_+^1\big)$ (resp.~${\bf x}_j^+$ for $\sN_k(a_-^{r},a_+^r)$) satisfy ${\bf x}_j^-={\bf y}_j+c_L$ (resp.~${\bf x}_j^+={\bf y}_j+c_R$).

 The $kr\times kr$ matrix $\Phi$ defining the asymptotic metric is described as follows. Let $\Phi^i$, $i=1,\dots,r-1$, be the $2k\times 2k$ matrix describing the metric on $\sN_k(a_-^{i},a_+^{i})$. We decompose each $\Phi^i$ into $k\times k$ blocks (corresponding to the positive and negative superscripts labelling coordinates) as
\[
\begin{pmatrix} \Phi_{11}^i & \Phi^i_{12}\\
 \Phi_{21}^i & \Phi^i_{22}\end{pmatrix}.
 \]
Next, we form an $rk\times rk$-matrix $\Psi^i$ as follows: the matrix $\Psi^i$ has $k^2$ $r\times
 r$ blocks labelled by~$\Psi^i_{(m,n)}$, where, for $i\leq r-1$,
\[
\Psi^i_{(m,n)}=\begin{cases} \Phi_{st}^i & \text{if $m=i+s-2$ and $n=i+t-2$},\\
0& \text{otherwise}.\end{cases}
\]
 For $i=r$, set $\Psi^r_{(r,r)}=\Phi^r_{11}$, $\Psi^r_{(r,1)}=\Phi^r_{12}$, $\Psi^r_{(1,r)}=\Phi_{21}^r$, $\Psi^r_{(1,1)}=\Phi_{22}^r$, and the remaining blocks equal to $0$. Finally, let $\Psi^0$ have the $(1,1)$-block equal to $\diag\bigl(\|{\bf y_1}\|^{-1},\dots, \|{\bf y_k}\|^{-1}\bigr)$, and all other blocks equal to $0$. Then the matrix $\Phi$ for the asymptotic metric is the sum $\sum_{j=0}^r \Psi^j$ with ${\bf x}_j^-$ for $\sN_k(a_-^{1},a_+^1)$ and ${\bf x}_j^+$ for $\sN_k(a_-^{r},a_+^r)$ replaced by, respectively, ${\bf y}_j+c_L$ and ${\bf y}_j+c_R$.

 To recapitulate: the asymptotic metric is given by formula \eqref{metric} for the just defined $rk\times rk$ matrix $\Phi$ in coordinates ${\bf y}_1,\dots,{\bf y_k}$, ${\bf x}_{ij}$, $i=1,\dots, r-1$, $j=1,\dots,k$.
\begin{Remark} The asymptotic metric appears already, albeit in a different form, in \cite[Section~9]{Che2}. The setup we have just presented allows to prove easily that it is, indeed, the asymptotic metric on~$\sM$.
\end{Remark}

 Let now
 \[
 R=\min\bigl\{\|{\bf y}_m-{\bf y}_n\|,\|{\bf x}_{im}-{\bf x}_{in}\|;\: i=1,\dots,r-1,\; m,n=1,\dots,k,\; m\neq n\bigr\}.
 \]
 If $R\to \infty$, then the spectral curves become close to unions of lines.
 The proof that this metric is exponentially (in the parameter $R$) close to the metric on $\sM$ proceeds as in \cite[Theorem~9.1]{BSU}, with minor modifications (the main one being that we can solve the real Nahm equation with boundary conditions of $\sM$ since $R>0$ guarantees that the stability condition for the complex-symplectic quotient of $\tilde\sM\times T^\ast \Mat_{k,k}(\cx)$ (cf.\ Section~\ref{bow}) is satisfied).

\pdfbookmark[1]{References}{ref}
\LastPageEnding

\end{document}